\documentclass[10pt]{amsart}
\usepackage{color}
\usepackage{amssymb}
\usepackage{epsfig}
\usepackage{url}
\usepackage{setspace}
\theoremstyle{plain}

\newtheorem{thm}{Theorem}[section]
\newtheorem{cor}[thm]{Corollary}
\newtheorem{defi}[thm]{Definition}

\newtheorem{prop}[thm]{Proposition}
\newtheorem{rem}[thm]{Remark}
\newtheorem{ques}[thm]{Question}
\newtheorem{conj}[thm]{Conjecture}
\newtheorem{exam}[thm]{Example}
\newtheorem{prob}[thm]{Problem}
\def\cal{\mathcal}
\def\bbb{\mathbb}
\def\op{\operatorname}
\renewcommand{\phi}{\varphi}

\newcommand{\N}{\bbb{N}}
\newcommand{\Z}{\bbb{Z}}
\newcommand{\Q}{\bbb{Q}}

\makeatletter
\let\@@pmod\pmod
\DeclareRobustCommand{\pmod}{\@ifstar\@pmods\@@pmod}
\def\@pmods#1{\mkern4mu({\operator@font mod}\mkern 6mu#1)}
\makeatother

\begin{document}
\title[Geometric progressions in the sets of values of rational functions]{Geometric progressions in the sets of values of rational functions}

\author{Maciej Ulas}

\thanks{The research of the authors is supported by the grant of the National Science Centre (NCN), Poland, no. UMO-2019/34/E/ST1/00094}

\maketitle
\begin{abstract}
Let $a, Q\in\Q$ be given and consider the set $\cal{G}(a, Q)=\{aQ^{i}:\;i\in\N\}$ of terms of geometric progression with 0th term equal to $a$ and the quotient $Q$. Let $f\in\Q(x, y)$ and $\cal{V}_{f}$ be the set of finite values of $f$. We consider the problem of existence of $a, Q\in\Q$ such that $\cal{G}(a, Q)\subset\cal{V}_{f}$. In the first part of the paper we describe several classes of rational function for which our problem has a positive solution. In particular, if $f(x,y)=\frac{f_{1}(x,y)}{f_{2}(x,y)}$, where $f_{1}, f_{2}\in\Z[x,y]$ are homogenous forms of degrees $d_{1}, d_{2}$ and $|d_{1}-d_{2}|=1$, we prove that $\cal{G}(a, Q)\subset \cal{V}_{f}$ if and only if there are $u, v\in\Q$ such that $a=f(u, v)$. In the second, experimental, part of the paper we study the stated problem for the rational function $f(x, y)=(y^2-x^3)/x$. We relate the problem to the existence of rational points on certain elliptic curves and present interesting numerical observations which allow us to state several questions and conjectures.
\end{abstract}

\begin{center}
Dedicated to Jaap Top on the occasion of his $60+2+\varepsilon$ birthday.
\end{center}

\section{Introduction and motivation}\label{sec1}

The problem of finding all integer solutions of the Diophantine equation
$$
x^2+7=2^n
$$
is a classical one. In fact, the above equation is the famous Ramanujan-Nagell equation. Ramanujan conjectured that the equation has only five solutions corresponding to $n=3, 4, 5, 7, 15$ \cite{Ram}. This was proved by Nagell in \cite{Nag}. Today, by a Ramanujan-Nagell equation we understand a general equation of the form
$$
x^2+D=aQ^n,
$$
where $a, Q, D\in\Z$ are fixed. The theory of such equations is quite well developed and we have many explicit results concerning form of solutions of particular equations or the number of integer solutions (see for example \cite{Beu1, BS} and references given therein). It is clear that one can consider even more general equation of the form
\begin{equation}\label{GRNequation}
f(x)=aQ^{n},
\end{equation}
where $f$ is a given polynomial with rational coefficients and $a, Q\in\Z$ are fixed. We know that if $f$ has at least two roots in $\overline{\Q}$, then (\ref{GRNequation}) has only finitely many solutions in integers $x, n$. In fact, even if we allow $Q$ to vary, then, under mild conditions on $f$, the equation (\ref{GRNequation}) has only finitely many solutions in integers $x, Q, n$. More on this line of research, or to be more precise, the realm of this type of Diophantine equations, can be found in the monograph \cite{ShoTij}.

Let us observe that the result concerning the solvability of (\ref{GRNequation}), where $a, Q$ are fixed, can be phrased in a different way. More precisely, let us consider the set
$$
\cal{G}(a, Q)=\{aQ^{i}:\;i\in\N\},
$$
where $a, Q\in\Q, aQ(Q^2-1)\neq 0$, i.e., an infinite set of values of geometric progression with initial term equal to $a$ and the quotient equal to $Q$, and the set $\cal{W}_{f}=\{f(x):\;x\in\Z\}$ - the set of values of the polynomial $f$. In other words, the equation (\ref{GRNequation}) is solvable in integers if and only if $\cal{G}(a, Q)\cap \cal{W}_{f}$ is non-empty. From our discussion it follows that if $f\in\Q[x]$, then the only possibility for $\cal{G}(a, Q)\cap \cal{W}_{f}$ to be infinite is when $f$ is a power of a linear polynomial. Moreover, if $f$ is linear polynomial then it is not difficult to construct infinitely many pairs $a, Q\in\Z$ such that the stronger condition $\cal{G}(a, Q)\subset \cal{W}_{f}$ holds. Indeed, let $f(x)=Ax+B$, where $A, B$ are fixed and observe that for each $t, T\in\Z$ we have $\cal{G}(f(t),AT+1)\subset \cal{W}_{f}$. This is a consequence of the identity $f(v_{n})=f(t)(AT+1)^{n}$, where $v_{0}=f(t)$ and for $n\geq 1$ we have $v_{n}=(AT+1)v_{n-1}+BT$.

Let us change a perspective and consider polynomials with more variables. Then, it is possible that $\cal{G}(a, Q)$ is contained in the set of values of $f$ even if the degree of $f$ is $\geq 2$. For example, if $p\equiv 1\pmod*{4}$ is a prime number, then the set $\cal{G}(1, p)=\{p^{n}:\;n\in\N\}$ is contained in the set of values of the polynomial $x^2+y^2$. This simple observation motivated us to consider, for a given $f\in\Q(x,y)$, the set
$$
\cal{V}_{f}=\{f(u,v)\in\Q:\;u, v\in\Q\;\mbox{for which}\;f(u,v)\;\mbox{determined}\}
$$
and study the following general problem.

\begin{prob}\label{mainprob}
Let $f\in\Q(x,y)$ be given. Does there exist rational numbers $a, Q$ such that $\cal{G}(a, Q)\subset \cal{V}_{f}$?
\end{prob}

Let us note that the above problem in certain cases can be seen as a variant of a well studied problem. More precisely, if $f\in\Z[x,y]$ is a homogenous form of degree $d\geq 3$, then for given $a, Q\in\Z$ the equation $f(x,y)=aQ^{i}$ is a special case of a so called Thue-Mahler equation. It particular, it has only finitely many solutions $(x, y, i)$ in co-prime integers $x, y$ and $i\in\N$. This was proved by Mahler in \cite{Mah} (generalizing an earlier result of Thue \cite{Th} concerning the equation $f(x,y)=A$, where $A\in\Z\setminus\{0\}$ is fixed). However, we do not assume that $x, y$ are integers. This assumption makes the problem more interesting and, in certain cases, allows to use the techniques typically applied in Diophantine problems where there is no real restriction between rational and integral solutions (like in case of homogenous problems).

Let us describe the content of the paper in some details. In Section \ref{sec2} we present several classes of rational functions for which Problem \ref{mainprob} has a solution. We also introduce a distinction between solutions of Problem \ref{mainprob} in the case of weighted homogenous functions of degree $d$. In Section \ref{sec3} we give some observations concerning Problem \ref{mainprob} for the rational function $f(x,y)=(y^2-x^3)/x$ and shows connections with the existence of rational points on certain elliptic curve (Theorem \ref{fixeda}). As a consequence of our investigations, we prove that for each $a\in\{1, \ldots, 10^{3}\}$ such that the elliptic curve $y^2=x^3+ax$ has positive rank, there are infinitely many values of $Q\in\Q$ such that $Q$ is not a square and $\cal{G}(a, Q)\subset \cal{V}_{f}$. Finally, in the last section we consider Problem \ref{mainprob} for the rational function $f(x,y)=(y^2-x^3)/x$ from a computational point of view. We present results of various computer experiments and formulate several questions, problems and conjectures which may stimulate further research .

During our computational experiments we extensively used two computational packages: the latest version of PARI/GP \cite{pari} (in case of rank computations) \cite{pari} and Mathematica 13.2 (for symbolic computations and data manipulation) \cite{Wol}.

\section{Basic results}\label{sec2}

In this section we present certain classes of rational functions such that for each element $f$ of an appropriate class there are values $a, Q\in\Q$ such that $\cal{G}(a, Q)\subset \cal{V}_{f}$.

\begin{thm}\label{class1}
Let $g_{1}, g_{2}, h_{1}, h_{2}\in\Q(y)$ and suppose that $g_{1}h_{2}\neq h_{1}g_{2}$ in $\Q(y)$. Let
$$
f(x,y)=\frac{xg_{1}(y)+g_{2}(y)}{xh_{1}(y)+h_{2}(y)}.
$$
Then, for any $a\in\Q\setminus\{0\}$ and $Q\neq \pm 1$ we have $\cal{G}(a, Q)\subset \cal{V}_{f}$.
\end{thm}
\begin{proof}
Let $a, Q$ be given. For given $i\in\N$ we need to find rational numbers $x_{i}, y_{i}$ such that $f(x_{i}, y_{i})=aQ^{i}$ for $i=0, 1, \ldots$. From our assumptions, for any $i\in\N$, we can chose $y_{i}\in\Q$ such that $aQ^{i}h_{1}(y_{i})-g_{1}(y_{i})\neq 0$. Then, it is enough to take
$$
x_{i}=\frac{g_{2}(y_{i})-aQ^{i}h_{2}(y_{i})}{aQ^{i}h_{1}(y_{i})-g_{1}(y_{i})}
$$
and note that $f(x_{i}, y_{i})=aQ^{i}$ for $i\in\N$.
\end{proof}

\begin{thm}\label{class2}
Let $f(x,y)=ux^2+vxy+wy^2\in\Q[x,y]\setminus\{0\}$. There are infinitely many pairs $a, Q\in\Q$ such that $\cal{G}(a, Q)\subset \cal{V}_{f}$ is infinite.
\end{thm}
\begin{proof}
Multiplying $f$ by $4u$ we get that $4uf(x,y)=(2ux+vy)^2+(4uw-v^2)y^2$. We substitute $a=4ua'$ and write $X=2ux+vy, Y=y, d=4uw-v^2$. Thus, the problem of finding $a, Q\in\Q$ such that $\cal{G}(a, Q)\subset \cal{V}_{f}$ is equivalent with the problem of finding $a', Q\in\Q$ such that $\cal{G}(a', Q)\subset \cal{V}_{F}$, where $F(x,y)=x^2+dy^2$. To get the result we take $a'=r^2+ds^2, Q=u^2+dv^2$, where $r, s, u, v\in\Q$ are chosen in such a way that $rsuvF(r, s)F(u, v)\neq 0$. Then, noting that the form $F$ is multiplicative, we get the result.
\end{proof}

\begin{thm}\label{class3}
Let $f(x,y)=\frac{f_{1}(x,y)}{f_{2}(x,y)}$, where $f_{1}, f_{2}\in\Z[x,y]$ are homogenous forms of degrees $d_{1}, d_{2}$, respectively. Let us assume that $|d_{1}-d_{2}|=1$. Then, for any $Q\in\Q\setminus\{-1,0, 1\}$ we have $\cal{G}(a, Q)\subset \cal{V}_{f}$ if and only $a=f(u, v)$ for some $u, v\in\Q$.
\end{thm}
\begin{proof}
Without loss of generality we can assume that $d_{1}-d_{2}=1$. If $\cal{G}(a, Q)\subset \cal{V}_{f}$, then for some $u, v\in\Q$ we have $a=f(u, v)$. To get the implication in the other side, we take $x_{i}=uQ^{i}, y_{i}=vQ^{i}$, where $u, v$ are rational parameters. Then we have
$$
f(x_{i}, y_{i})=\frac{f_{1}(uQ^{i}, vQ^{i})}{f_{2}(uQ^{i}, vQ^{i})}=\frac{Q^{d_{1}i}f_{1}(u, v)}{Q^{d_{2}i}f_{2}(u, v)}=Q^{i}f(u,v).
$$
Thus, if $a=f(u, v)$ for some $u, v\in\Q$ such that $f(u, v)\neq 0$ we have $f(x_{i}, y_{i})=aQ^{i}$ for each $i\in\N$. Consequently $\cal{G}(a, Q)\subset \cal{V}_{f}$.
\end{proof}

\begin{rem}\label{rem1}
{\rm Theorem \ref{class3} is a special (easily) provable part of the following general observation. Suppose that $f\in\Q(x, y)$ is weighted homogeneous of degree $d$, i.e., there is a pair $(w_{1}, w_{2})$ of integers, called weights, with such a property that for a variable $\lambda$, we have $f(\lambda^{w_{1}}x, \lambda^{w_{2}}y)=\lambda^{d} f(x,y)$. Thus, to find $a, Q\in\Q$ such that $\cal{G}(a, Q)\subset \cal{V}_{f}$ it is enough to prove that we have $\{a, aQ, \ldots, aQ^{d-1}\}\subset \cal{V}_{f}$. This is clear, if $f(x_{i}, y_{i})=aQ^{i}$ for $i=0, 1, \ldots, d-1$, then $f(x_{i}Q^{w_{1}j}, y_{i}Q^{w_{2}j})=aQ^{i+dj}$ for $j\in\Z$ and thus $\cal{G}(a, Q)\subset\cal{V}_{f}$.

Note that a rational function $f$ is homogenous of degree $d$ if and only if is weighted homogenous of degree $d$ with weights $(w_{1}, w_{2})=(1, 1)$.
}
\end{rem}

\begin{exam}
{\rm Let $f(x, y)=x^3+y^3$. We find that $\cal{G}(13797, 2)\subset \cal{V}_{f}$ which follows from observation made in Remark \ref{rem1} and the equalities
$$
f(-3, 24)=13797,\quad f(-13, 31)=2\cdot 13797,\quad  f(-33, 45)=2^2\cdot  13797.
$$

Let us also note that the problem of finding $a, Q$ such that $\cal{G}(a, Q)\subset \cal{V}_{f}$ can be related to the Problem \ref{mainprob} with different rational function $f$. Indeed, let us recall that the curve $C:\;x^3+y^3=A$ is birationally equivalent with the elliptic curve in Weierstrass form $E:\;y^2=x^3-432A^2$. We have the maps
\begin{align*}
\phi:&\;C\ni (x,y)\mapsto \left(\frac{36 A+y}{6 x},\frac{36 A-y}{6 x}\right)\in E,\\
\psi:&\;E\ni (x,y)\mapsto \left(\frac{12 A}{x+y},\frac{36 A (y-x)}{x+y}\right)\in C.
\end{align*}
In other words, if $g(x,y)=y^2-x^3$, then we have
$$
\cal{G}(a, Q)\subset \cal{V}_{f} \Longleftrightarrow  \cal{G}(-432a^2, Q^2)\subset \cal{V}_{g}.
$$

Let us consider $a=6, Q=7$, i.e., for $i=0, 1, 2$ we consider the curve $E_{i}:\;y^2=x^3-432\cdot 6^2\cdot 7^{2i}$. One can check that the point $S_{i}$ lies on the curve $E_{i}$ for $i=0, 1, 2$, where
$$
S_{0}=(28, 80), \;S_{1}=(172, 2080),\;S_{2}=(2353, 113975).
$$
Then, we have $x_{i}^3+y_{i}^3=6\cdot 7^{i}$ for the corresponding points $P_{i}=\psi(S_{i})=(x_{i}, y_{i})$, where
$$
P_{0}=\left(\frac{37}{21},\frac{17}{21}\right),\; P_{1}=\left(\frac{449}{129},-\frac{71}{129}\right),\; P_{2}=\left(\frac{124559}{14118},-\frac{103391}{14118}\right).
$$

As a consequence we get the following

\begin{cor}\label{cor1}
Let $\cal{C}=\{x^3+y^3:\;x, y\in\Z\}\subset \Z$. The set $\cal{C}$ contains arbitrarily long geometric progressions.
\end{cor}
\begin{proof}
Let $n\in\N_{+}$ be given and $f(x,y)=x^3+y^3$. We proved that the set $\cal{G}(6, 7)\subset \cal{V}_{f}$. Let $x_{i}, y_{i}\in\Q\setminus\{0\}$ satisfy $f(x_{i}, y_{i})=6\cdot 7^{i}$ for $i\in\N$. Let $D_{n}$ be a non-zero integer such that $D_{n}\prod_{i=0}^{n}(x_{i}y_{i})\in\Z$. Then the numbers $X_{i}=D_{n}x_{i}, Y_{i}=D_{n}y_{i}$ are integers and
we have
$$
f(X_{i}, Y_{i})=6D_{n}^37^{i}\quad\mbox{for}\quad i=0, 1, \ldots, n,
$$
i.e., the set $\cal{C}$ contains geometric progression $\{6D_{n}^37^{i}:\;i=0,\ldots, n\}$ and the result follows.
\end{proof}

In the light of the proof of Corollary \ref{cor1} on can ask the following

\begin{ques}
Do there exist integers $a, Q$ such that for each $i\in\N$ there is a co-prime solution in integers $x, y$ of the Diophantine equation $x^3+y^3=aQ^{i}$?
\end{ques}

}
\end{exam}

We should also note that in the case of $f\in\Q(x,y)$ which is weighted homogenous of degree $d$ with weights $(w_{1}, w_{2})$, the Problem \ref{mainprob} has a trivial solution. Indeed, for each $Q\in\Q, Q(Q^2-1)\neq 0$, and any $u, v\in\Q$ such that $f(u,v)\neq 0$, we have $\cal{G}(f(u,v), Q^{d})\subset \cal{V}_{f}$. This is simple consequence of the identity
$$
f(uQ^{iw_{1}},vQ^{iw_{2}})=f(u,v)Q^{id}.
$$
Thus, in case of weighted homogenous rational function $f$ of degree $d$, we need to consider various notions of solutions of Problem \ref{mainprob}.
\begin{defi}
Let $f\in\Q(x, y)$ be weighted homogenous of degree $d$. We say that Problem \ref{mainprob} has
\begin{enumerate}
\item a trivial solution $(a, Q)$ if and only if $Q$ is a $d$-th power of a rational number and $\cal{G}(a, Q)\subset \cal{V}_{f}$;
\item a non-trivial solution $(a, Q)$ of level $e$ if and only if $e|d, e<d, Q$ is $e$-th power of a rational number, there is at least one prime $q$ such that the $q$-adic valuation of $Q$ is equal to $e$ (or $-e$), and $\cal{G}(a, Q)\subset \cal{V}_{f}$;
\item a proper solution $(a, Q)$ if and only if $(a, Q)$ is a non-trivial solution of level $1$.
\end{enumerate}
\end{defi}

We already observed that for each weighted homogenous rational function $f$, Problem \ref{mainprob} has a trivial solution. On the other hand, for each $f$ from the class of homogenous rational functions of degree 1, we know, via Theorem \ref{class3}, that Problem \ref{mainprob} has a proper solution. The same is true for the function $f(x, y)=x^3+y^3$. In the next section we will investigate the existence and properties of proper solutions of Problem \ref{mainprob} for $f(x, y)=(y^2-x^3)/x$. However, before we do that we note the following.

%let us note that if $d\in\N_{+}$ is odd and $f(x, y)=(y^2-x^d)/x$ then for each $Q\in\Q\setminus\{-1, 0, 1\}$ which is not a square, there are infinitely many values of $a\in\Z$ %such that $(a, Q^{d-1})$ is non-trivial solution of Problem \ref{mainprob} of level $d-1$$.

% has a non-trivial solution, if and only if there is a non-zero $a\in \Q$ and $Q\in\Q\setminus\{-1,0, 1\}$ such that $Q$ is not a $d$-th power and $\cal{G}(a, Q)\subset %\cal{V}_{f}$. In the next result we show that for certain weighted homogenous rational functions with arbitrarily large degree one can find non-trivial solution of Problem %\ref{mainprob}.

\begin{thm}\label{bihomo}
Let $d\in\N$ be odd.
\begin{enumerate}
\item Let $f(x,y)=(y^2-x^d)/x$. For each $Q\in\Q\setminus\{-1\}$ which is not a square, there are infinitely many values of $a\in\Z$ such that $(a, Q^{d-1})$ is a non-trivial solution of level $d-1$ of  Problem \ref{mainprob}.
\item Let $f(x,y)=y^2-x^d$. For each $Q\in\Q\setminus\{-1\}$ which is not a square, there are infinitely many values of $a\in\Z$ such that $(a, Q^d)$ is a non-trivial solution of level $d$ of  Problem \ref{mainprob}.
\end{enumerate}
\end{thm}
\begin{proof}
We start with the case $f(x,y)=(y^2-x^d)/x$. In this case $f$ is weighted homogenous of degree $2(d-1)$ with weights $(2, d)$.
Let us note that to get the statement $\cal{G}(a, Q^{d-1})\subset \cal{V}_{f}$ it is enough to prove that there exist rational solutions of the system of equations
$$
f(x_{0}, y_{0})=a,\quad f(x_{1}, y_{1})=aQ^{d-1}
$$
in variables $a, x_{i}, y_{i}, i=0, 1$. Equivalently, by eliminating $a$, we consider the equation
$$
f(x_{1}, y_{1})=f(x_{0}, y_{0})Q^{d-1}.
$$
It can be easily solved via the substitution $x_{i}=p_{i}T, y_{i}=q_{i}T^{m}$, where $m=(d-1)/2, p_{i}, q_{i}$ for $i=0, 1$ are rational variables and we look for an appropriate specialization of $T$. Indeed, after the substitution, our equation reduces to the equation (in one variable $T$) of the form
$$
T^{d}[p_{0}p_{1}(p_{1}^{d-1}-p_{0}^{d-1}Q^{d-1})T+p_{1}q_{0}^2Q^{d-1}-p_{0}q_{1}^2]=0.
$$
Solving the equation defined by the second factor, we get an appropriate value of $T$, together with the corresponding value of $a$ in the following form
$$
T=\frac{p_{0}q_{1}^2-p_{1}q_{0}^2Q^{d-1}}{p_{0}p_{1}(p_{1}^{d-1}-p_{0}^{d-1}Q^{d-1})}, \quad a=f(x_{0}, y_{0})=f(p_{0}T, q_{0}T^{m})=\frac{T^{d-2}(q_{0}^2-p_{0}^{d}T)}{p_{0}}.
$$
As a consequence of our reasoning we see that by taking $X_{i}=x_{i\pmod*{2}}Q^{2\lfloor\frac{i}{2}\rfloor}, Y_{i}=x_{i\pmod*{2}}Q^{d\lfloor\frac{i}{2}\rfloor}$ we get $f(X_{i}, Y_{i})=aQ^{(d-1)i}$ for $i\in\N$, and thus $\cal{G}(a, Q^{d-1})\subset \cal{V}_{f}$. If $Q$ is not a square then our solution is non-trivial.

\bigskip

To get the statement for $f(x,y)=y^2-x^d$ we note that $f$ is weighted homogenous of degree $2d$ with weights $(2, d)$. It is clear that to prove that $\cal{G}(a, Q^{d})\subset\cal{V}_{f}$ it is enough to solve the system
$$
f(x_{0}, y_{0})=a,\quad f(x_{1}, y_{1})=aQ^{d}.
$$
or equivalently, the equation $f(x_{1}, y_{1})=f(x_{0}, y_{0})Q^{d}$. Using the same substitutions and the same reasoning as in the case of the function $(y^2-x^d)/x$ we get the corresponding values of $T$ and $a$ in the following form
$$
T=\frac{q_{0}^2-Q^dq_{1}^2}{p_{0}^d-Q^{d}p_{1}^d},\quad a=f(x_{0}, y_{0})=f(p_{0}T, q_{0}T^{m})=T^{d-1}(q_{0}^2-p_{0}^{d}T).
$$
\end{proof}

\section{Geometric progressions in the set $\cal{V}_{f}$ for $f(x,y)=\frac{y^2-x^3}{x}$ and related elliptic curves}\label{sec3}

In the previous section we obtained some general result concerning Problem \ref{mainprob}. Now we concentrate on a specific rational function. Let us put
$$
f(x,y)=\frac{y^2-x^3}{x}.
$$
We know that $f$ is weighted-homogenous of degree 4 with weights $(2, 3)$ and from the first part of  Theorem \ref{bihomo}, for each $Q\in\Q$ which is not a square we have $\cal{G}(a, Q^{2})\subset \cal{V}_{f}$ non-trivially. Before we will go on, let us note that without loss of generality one can assume that $a, Q\in\Z$. Indeed, if $a=p/q, Q=u/v$, then by multiplication of the equality $y^2=x^3+aQ^{i}x$ by $q^6v^{6i}$ we easily get the equivalence
$$
\cal{G}(a, Q)\subset \cal{V}_{f}\;\Longleftrightarrow\;\cal{G}(pq^{3}, uv^{3})\subset \cal{V}_{f}.
$$
Moreover, let us note that for any $t\in\Q$ we have: $(a, Q)$ is a solution of Problem \ref{mainprob} if and only if $(at^{4}, Q)$ is a solution of Problem \ref{mainprob}.

Although the choice of $f$ may be seen arbitrary, it is not. This follows from the fact that for any given $A\neq 0$ the curve $f(x,y)=A$ is an elliptic curve, and the existence of $a, Q$ such that $\cal{G}(a, Q)\subset \cal{V}_{f}$ is equivalent with the property that for each $i\in\N$, the elliptic curve
$$
E_{i}(a, Q):\;y^2=x^3+aQ^{i}x
$$
contains rational point with $x\neq 0$. Actually, in the sequel we require that the curve has a positive rank. This connection with elliptic curves suggest several interesting questions.
% To make the problem more interesting, we will consider only integers $Q$ with such a property that for at least one prime $p$ dividing $Q$, we have $p^2\nmid Q$. Under this %assumption, taking into account that $f(x, y)$ is weighted homogenous of degree 4, we see that $\cal{G}(a, Q)\subset\cal{V}_{f}$ if and only if the elliptic curve $E_{i}(a, Q)$ %has a positive rank for $i=0, 1, 2, 3$.

%We discuss this line of research in the next section.

It is not difficult to find an example of rational numbers $a, Q$ such that $\cal{G}(a, Q)$ is infinite and $\cal{G}(a, Q)\subset \cal{V}_{f}$. Indeed, let us take $Q=2$ and $a=47$. We thus consider elliptic curves $E_{i}(47, 2):\; y^2=x^3+47\cdot 2^{i}x$. One can compute the rank of $E_{i}(47, 2)$ and the set of generators of infinite order. We collect the computed data in the table below.
\begin{equation*}
\begin{array}{|c|c|l|}
\hline
  i & r(E_{i}(47, 2)) & \mbox{generators of}\;E_{i}(47, 2)(\Q) \\
\hline
  0 & 1 & (289/25, -5712/125) \\
  1 & 2 & (2, 14), (1504/81, 65800/729) \\
  2 & 1 & (18, 96)  \\
  3 & 1 & (4716544/18225, 10271916928/2460375) \\
\hline
\end{array}
\end{equation*}
\begin{center}
Table 1. Ranks and the generators for the elliptic curves $E_{i}(47, 2):\;y^2=x^3+47\cdot2^{i}x, i=0, 1, 2, 3$.
\end{center}

As a consequence we get that $\cal{G}(47, 2)\subset \cal{V}_{f}$. In the light of this result one can consider related problem. More precisely, let $a\in\Q\setminus\{0\}$ be given and assume that the elliptic curve $E:\;y^2=x^3+ax$ has a positive rank. Is it possible to find a $Q\in\Z\setminus\{-1,0,1\}$ which is fourth power-free and the elliptic curve $E_{i}(a, Q)$ has positive rank for $i=0, 1, 2, 3$? We prove the following:

\begin{thm}\label{fixeda}
Let $a\in\Z\setminus\{0\}$ and suppose that the elliptic curve $E:\;y^2=x^3+ax$ has a positive rank. Let us consider the elliptic curve
$$
\cal{E}_{a}(p, v):\;Y^2=X(X-f_{1}(a,p,v))(X-f_{2}(a,p,v))=:F(X),
$$
where
\begin{align*}
f_{1}&=ap^3(p^2+a)^3v^4,\\
f_{2}&=a p (p^2+a) ((p^2+a)pv^2-1)((p^2+a)pv^2+1).
\end{align*}
If there are $p, v\in\Q$ such that $\cal{E}_{a}(p, v)$ has a positive rank, then there are infinitely many values of fourth power-free integers $Q$ which are not squares and such that $\cal{G}(a, Q)\subset \cal{V}_{f}$.
\end{thm}
\begin{proof}
It is clear that we are interested in finding solutions of the following system of Diophantine equations
\begin{equation}\label{Qsys}
\frac{y_{1}^2-x_{1}^3}{ax_{1}}=Q,\quad \frac{y_{2}^2-x_{2}^3}{ax_{2}}=Q^2,\quad \frac{y_{3}^2-x_{3}^3}{ax_{3}}=Q^3
\end{equation}
in variables $x_{i}, y_{i}, i=1, 2, 3$ and  $Q$. To find a solution we make the following substitution
\begin{equation}\label{Qsub1}
x_{1}=v^2,\;y_{1}=y,\quad x_{2}=pQ, \;y_{2}=qQ^2,\quad x_{3}=rQ^2,\;y_{3}=sQ^3
\end{equation}
which reduces the system (\ref{Qsys}) to the system
\begin{equation}\label{Qsys1}
Q=g_{1}(v, y)=g_{2}(p,q)=g_{3}(r,s),
\end{equation}
where
\begin{equation*}
g_{1}(v, y)=\frac{y^2-v^6}{av^2},\quad g_{2}(p, q)=\frac{p^3+ap}{q^2},\quad g_{3}(r,s)=\frac{ar}{s^2-r^3}.
\end{equation*}
We take $r=ap(p^2+a)$ and observe that the equation $g_{2}(p, q)=g_{3}(r, s)$ defines a quadric in variables $(q, s)$ which contains rational points at infinity and can be parameterized in the following way
$$
q=-\frac{a^3 p^3 u^2-a^3-3 a^2 p^2-3 a p^4-p^6}{2 a u},\quad s=\frac{a^3 p^3 u^2+a^3+3 a^2 p^2+3 a p^4+p^6}{2 u}.
$$
With $q, r, s$ chosen in this way we left with one equation $g_{1}(v, y)=g_{2}(p, q(p, u))$. If we treat $p, v$ as a rational independent variables, then this equation defines the curve in the plane $(u, Y)$, over the field $\Q(p, v)$, of the form
\begin{align*}
\cal{C}_{a}:\; (&a^3 p^3 u^2-(p^2+a)^3)^2y^2=\\
                &a^6p^6v^6u^4-2a^3p(p^2+a)(p^2(p^2+a)^2v^4-2)v^2u^2+(p^2+a)^6v^6=:F_{a}(p, v, u).
\end{align*}
%\begin{equation*}
%\cal{C}_{a}:\; Y^2=a^6p^6v^6u^4-2a^3p(p^2+a)(p^2(p^2+a)^2v^4-2)v^2u^2+(p^2+a)^6v^6=:F_{a}(p, v, u),
%\end{equation*} where we put $Y=(a^3 p^3 u^2-(p^2+a)^3)y$.
As a consequence of our reasoning, we see that each point $(u, Y)$ lying on the curve $\cal{C}_{a}$, together with corresponding values of $q, r, s$ defined above, lead to the value of $Q$ in the following form
$$
Q=Q_{a}(p, u)=\frac{4a^2p(p^2+a)u^2}{(a^3 p^3 u^2-(p^2+a)^3)^2}.
$$

The genus of the curve $\cal{C}_{a}$ is equal to 1. Moreover, $\cal{C}_{a}$ contains $\Q(p, v)$-rational point $(0,(p^2+a)^3v^3)$. Thus, the curve $\cal{C}_{a}$ is birationally equivalent with a curve given by an appropriate Weierstrass equation. In fact, $\cal{C}_{a}$ is birationally equivalent with the curve $\cal{E}_{a}:\;Y^2=X(X-f_{1}(p, v))(X-f_{2}(p, v))$ from the statement of our theorem, via an appropriate invertible map $\phi:\;\cal{C}_{a}\ni (u, Y)\mapsto (X, Y)\in\cal{E}_{a}$. We do not present the form of the map because it is quite complicated (but can be easily computed using standard method of finding birational model of a quartic curve of the form $y^2=Au^4+Bu^3+Cu^2+Du+q^2$ with rational point $(0, q)$, see \cite[Theorem 2.17]{Was}).

Let us observe that the discriminant of $F_{a}$ with respect to $u$ is given by
$$
D_{a}(p, v)=\op{Disc}_{u}(F_{a}(p, v, u))=2^{12}a^{18}p^{10}(a+p^2)^{10}v^{20}(p^2(p^2+a)^2v^4-1)^2.
$$
Because we are interested in values of $p=p_{0}, v=v_{0}$ such that $p_{0}(p_{0}^{2}+a)$ is not a square it is clear that $D_{a}(p_{0}, v_{0})\neq 0$ and the corresponding curve $\cal{C}_{a}(p_{0}, v_{0})$ is non-singular. This is also true for the curve $\cal{E}_{a}(p_{0}, v_{0})$.

If now $p_{0}, v_{0}\in\Q$ are chosen in such a way, that $f_{1}(p_{0}, v_{0})f_{2}(p_{0}, v_{0})\neq 0$ and the elliptic curve $\cal{E}_{a}(p_{0}, v_{0})$ obtained from $\cal{E}_{a}$ after specialization $p=p_{0}, v=v_{0}$, has infinitely many rational points, then we get infinitely many rational values of $Q$. However, to finish the proof we need to show that if $\cal{E}_{a}(p_{0}, v_{0})$ has infinitely many rational points, then we get infinitely many essentially different values of $Q$. More precisely, we prove that there is an infinite set $\cal{U}\subset\Q$ such that for each $u', u''\in\cal{U}$ the quotient $Q_{a}(p_{0}, u')/Q_{a}(p_{0}, u'')$ is not a fourth power. We will construct the sequence of finite sets $(\cal{U}_{n})_{n\in\N}$ recursively in the following way. Let us fix a point $P_{0}=(X_{0}',Y_{0}')$ of infinite order on $\cal{E}_{a}(p_{0}, v_{0})$ and write $(u_{n}, Y_{n})=\phi^{-1}(nP_{0})\in \cal{C}_{a}(p_{0}, v_{0})(\Q)$. Let us note that the point $(u, Y)\in\cal{C}_{a}(p_{0}, v_{0})$ such that $Q_{a}(p_{0}, u)/Q_{a}(p_{0}, u_{0})$ is a fourth power, lie on the variety
$$
\cal{W}_{a}(u_{0}):\;Q_{a}(p_{0}, u)=Q_{a}(p_{0}, u_{0})T^4,\quad Y^2=F_{a}(p_{0}, v_{0},u),
$$
which, as an intersection of two surfaces, defines a curve over $\Q$. One can check that the genus of $\cal{W}_{a}(u_{0})$ is equal to 5, and thus there are only finitely many $u$'s which give a rational point. Thus, there is a positive integer $N_{0}$ such that for all $N\geq N_{0}$ the quotient $Q_{a}(p_{0}, u_{n})/Q_{a}(p_{0}, u_{0})$ is not a fourth power. We thus take $\cal{U}_{0}=\{u_{0}', u_{1}'\}$, where $u_{0}'=u_{0}, u_{1}'=u_{N_{0}}$. Let us suppose that we constructed the set $U_{n-1}=\{u_{0}', u_{1}',\ldots, u_{n}'\}\subset\{u_{i}:\;i\in\N\}$ with required properties, i.e., for each $i, j\in\{0,\ldots, n\}, i\neq j$ the quotient $Q_{a}(p_{0}, v_{0}, u_{i}')/Q_{a}(p_{0}, v_{0}, u_{j}')$ is not a fourth power. To find $u_{n+1}'$ we note that for each $i\in\{0, \ldots, n\}$ the curve $\cal{W}_{a}(u_{i})$ has only finitely many rational solutions $(u, T, Y)$. Thus, there is an integer $N_{n+1}$ such that for all $n\geq N_{n+1}$ there is no rational numbers $T, Y$ such that $(u_{n}, T, Y)$ lies on $\cal{W}_{a}(u_{i})$ for $i=0, \ldots, n$. As a consequence of our reasoning we can take $u_{n+1}'=u_{N_{n+1}}$ and put $\cal{U}_{n+1}=\cal{U}_{n}\cup\{u_{n+1}'\}$. Our reasoning shows the existence of an increasing sequence $(\cal{U}_{n})_{n\in\N}$ of sets with cardinality $\# \cal{U}_{n}=n+1$, each having a required property. To finish the proof, it is enough to take $\cal{U}=\lim_{n\rightarrow+\infty}\cal{U}_{n}$.
\end{proof}

It should be noted that if $p_{0}, v_{0}\in\Q$ are such that $\cal{E}_{a}(p_{0}, v_{0})$ has positive rank, then the corresponding value of $Q$ tends to be pretty large. To see this, let us consider the following example.

\begin{exam}
{\rm Let us apply Theorem \ref{fixeda} for a specific value of $a$. We take $a=3$, which is the smallest positive integer such that the curve $E:\;y^2=x^3+3x$ has a positive rank. Next, we chose $p=4, v=2$ and the resulting elliptic curve is
$$
E_{3}(4,2):\;Y^2=X(X-21070848)(X-21070620).
$$
One can check that $E_{3}(4,2)$ has rank 1 and the generator of infinite part is generated by the point $P=(X,Y)=(732246016/9, 14683034857472/27)$. We have that
$$
\phi^{-1}(P)=(u, y)=\left(\frac{4204567}{4146944},-\frac{1592941018808}{199115595007}\right).
$$
Performing all necessary computations we found the value of $Q$ given by
$$
Q=Q_{3}\left(4, \frac{4204567}{4146944}\right)=19\cdot \left(\frac{476639376}{199115595007}\right)^2.
$$
Then we have the point $P_{i}$ of infinite order on $E_{i}(3, Q), i=0, 1, 2, 3$, where $P_{0}=(1, 2)$ and
\begin{align*}
P_{1}&=\left(4,\frac{1592941018808}{199115595007}\right),\\
P_{2}&=\left(\frac{17266067201278872576}{199115595007^2},\frac{78182031219520468152777449472}{199115595007^3}\right),\\
P_{3}&=\left(\frac{57\cdot 8633033600639436288^2}{39647020174991643330049^2}, \frac{127652133024668050546\cdot 51798201603836617728^2}{39647020174991643330049^3}\right).
\end{align*}
}
\end{exam}

\begin{rem}
{\rm Let us note that the smallest positive integer $Q>1$ such that $(3, Q)$ is a proper solution of Problem \ref{mainprob} is $Q=17$. The point $P_{i}$ is of infinite order on $E_{i}(3, 17), i=0, 1, 2, 3$, where
$$
P_{0}=(1, 1), \; P_{1}=(25, 130),\; P_{2}=\left(\frac{121}{25}, \frac{8206}{125}\right),\; P_{3}=\left(\frac{49}{121}, \frac{102830}{1331}\right).
$$
}
\end{rem}
Let
$$
\cal{R}=\{a\in\N_{+}:\;\mbox{rank of the curve}\;y^2=x^3+ax\;\mbox{is positive}\}.
$$
We performed small computer search and found that for each $a\in\cal{R}\cap[1, 10^3]$, there are $p_{0}, v_{0}\in\N$ such that the curve $\cal{E}_{a}(p_{0}, v_{0})$ has positive rank. In fact, in the considered range of $a$, we can always take $p_{0}$ from the set $\{2, 3, 4\}$ and $v_{0}$ form the set $\{1, \ldots, 10\}$. Although small, our computations suggest the following:

\begin{conj}\label{fixedacurve}
For each $a\in\cal{R}$ there is a specialization $p=p_{0}, v=v_{0}$ such that $p_{0}(p_{0}^2+a)$ is not a square and the rank of the elliptic curve $\cal{E}_{a}(p_{0}, v_{0})$ is positive.
\end{conj}

\bigskip

In light of Theorem \ref{fixeda} one can investigate the following related question. Let $Q\in\Z\setminus\{-1, 0, 1\}$ be given. Is it possible to find an $a\in\Z$ which is a fourth power-free and the elliptic curve $E_{i}(a, Q)$ has positive rank for $i=0, 1, 2, 3$? Equivalently, we are interested in the solutions of the system od Diophantine equations
\begin{equation}\label{asys1}
a=f(x_{0},y_{0})=\frac{f(x_{1},y_{1})}{Q}=\frac{f(x_{2},y_{2})}{Q^2}=\frac{f(x_{3},y_{3})}{Q^3}.
\end{equation}
The common value of the fractions $f(x_{i},y_{i})/Q^{i}, i=0, 1, 2, 3$ is the value of $a$ we are looking for. Let us put $x_{i}=p_{i}T, y_{i}=q_{i}T$ and solve the equation $f(x_{i}, y_{i})/Q^{i}=f(x_{i+1}, y_{i+1})/Q^{i+1}$ with respect to the variable $T$, for $i=0, 1, 2$.  As a result, we see that to solve (\ref{asys1}), the value of $T$ need to be of the form
$$
T=\frac{p_{0}q_{1}^2-Qp_{1}q_{0}^2}{p_{0}p_{1}(p_{1}^2-Qp_{0}^2)}=\frac{p_{1}q_{2}^2-Qp_{2}q_{1}^2}{p_{1}p_{2}(p_{2}^2-Qp_{1}^2)}=\frac{p_{2}q_{3}^2-Qp_{3}q_{2}^2}{p_{2}p_{3}(p_{3}^2-Qp_{2}^2)}.
$$
We arrived to the following conclusion: if $(x_{i}, y_{i}), i=0, 1, 2, 3$ solve the system (\ref{asys1}), then the corresponding values of $p_{i}, q_{i}, i=0, 1, 2, 3,$ satisfy the system
\begin{equation}\label{asys2}
\cal{C}'_{Q}:\;\begin{cases}
\begin{array}{lll}
  p_1 p_2 Q(p_1^2 Q-p_2^2)q_{0}^2+p_0 p_2(p_2^2-p_0^2 Q^2)q_{1}^2+p_0 p_1(p_0^2 Q-p_1^2)q_{2}^2 & = & 0, \\
  p_2 p_3 Q(p_2^2 Q-p_3^2)q_{1}^2+p_1 p_3(p_3^2-p_1^2 Q^2)q_{2}^2+p_1 p_2(p_1^2 Q-p_2^2)q_{3}^2 & = & 0,
\end{array}
\end{cases}
\end{equation}
and the corresponding value of $a$ has the form
$$
a=-\frac{(p_1^3 q_0^2-p_0^3 q_1^2)(p_1 q_0^2 Q-p_0 q_1^2)}{p_0^2p_1^2(p_1^2-p_0^2 Q)^2}.
$$
The variety  $\cal{C}'$ can be seen as a projective variety defined over the function field $K=\Q(Q,p_0,p_1,p_2,p_{3})$. This is just an intersection of two quadratic surfaces, and thus, from geometric point of view, we get a curve of genus 1. Unfortunately, it seems that there is no general point on $\cal{C}'$ and it is non-trivial problem to find a specialization $\overline{p}=(p_{0}, p_{1}, p_{2}, p_{3})$ such that the curve $\cal{C}'(\overline{p})$ has infinitely many rational points and $(a, Q)$ is a proper solution of Problem \ref{mainprob}. 

\begin{exam}
{\rm Let $Q=2$. We know that $(47, 2)$ is a proper solution of Problem \ref{mainprob} for $f(x, y)=(y^2-x^3)/x$. Using points from Table 1. one can take 
$\overline{p}=\left(289/25, 14,  96, 4716544/18225) \right)$ and note that the curve $\cal{C}'_{2}(\overline{p})$ contains the point $S=(-5712/125, 14, 96, 10271916928/2460375)$. Using this point as a point at infinity, we get that $\cal{C}'_{2}(\overline{p})$ is birational, via an appropriate map $\chi$, with the elliptic curve with minimal Weierstrass equation $\cal{E}'_{2}(\overline{p}):\;y^2=x^3+ax+b$, where
\begin{align*}
a&=-114789376213793508149518054796962396588,\\
b&=-154452043874358875364810179220092982963254111141598049488.
\end{align*}
The curve $\cal{E}'_{2}(\overline{p})$ contains the point of infinite order 
$$
\chi(S')=\left(-\frac{32968008403251534251}{9}, \frac{397625982502899147956961689375}{27}\right),
$$
where $S'=(5712/125, 14, 96, 10271916928/2460375)$. Thus, one can compute the points $\chi^{-1}(m\chi(S'))$ for $m\in\N_{+}, m\neq 2,$ and get corresponding points on  $\cal{C}'_{2}(\overline{p})$, and an appropriate value of $a_{m}(\overline{p},\overline{q})$. For $m=1$ we get nothing new, i.e., $a_{1}(\overline{p},\overline{q})=47u^4$ for some $u\in\Q$. For $m=3$, the value $a_{m}(\overline{p},\overline{q})$ is, up to fourth power, an integer with 1564 digits. 

As we will see in the next section, there are many $a$'s such that for given $Q\in\{2,\ldots, 12\}$ the pair is a proper solution of Problem \ref{mainprob}.
}
\end{exam}

\section{Computational observations, questions and conjectures}\label{sec4}

In this section we collect results of our computations and formulate some questions and conjectures which may stimulate further research. To compute ranks of several thousands of elliptic curves we used very fast procedure {\tt ellrank} implemented in PARI/GP \cite{pari}. The function {\tt ellrank} attempts to compute the rank of a given elliptic curve $E$. The function returns the vector $[r, R, s, L]$. The rank of $E(\Q)$ is between $r$ and $R$ (both included), $s$ is related to the Tate-Shafarevitch group, and $L$ is a list of independent, non-torsion rational points on the curve. To init the curve we need to use procedure {\tt ellinit}. Here is a sample session:

{\tt
\noindent gp> E=ellinit([47,0]); \\
\noindent gp> ellrank(E) \\
\noindent [1, 1, 0, [[289/25, 5712/125]]] \\
}
So the rank of an elliptic curve $E:\;y^2=x^3+47x$ is one and an infinite part of $E(\Q)$ is generated by the point $(289/25, 5712/125)$. However, in some cases PARI/GP is not able to compute the rank exactly, like in case of the elliptic curve $H:\;y^2=x^3+257x$:

{\tt
\noindent gp> H=ellinit([257,0]); \\
\noindent gp> ellrank(H) \\
\noindent [0, 2, 0, []] \\
}
We get that the rank of $H(\Q)$ is between 0 and 2. However, in such cases we invoke the procedure {\tt ellanalyticrank} which computes the pair $[r, d]$, where $r$ is an analytic rank of an elliptic curve and $d$ the value of $d$-th derivative at 1 of a $L$-function associated to $H$. We have

{\tt
\noindent gp> analyticrank(H) \\
\noindent gp> [0, 7.4090738601929029389865884843731744616] \\
}
Thus, invoking Kolyvagin's \cite{Kol} together with Wiles' \cite{Wil} result on modularity of elliptic curves over $\Q$, we know that the rank of $H$ is equal to 0.

After this digression, let us back to Problem \ref{mainprob} with $f(x, y)=(y^2-x^3)/x$. We performed additional computations and looked for positive integers $a, Q$ such that $\cal{G}(a, Q)\subset\cal{V}_{f}$ with $Q\in\{2,\ldots, 12\}$ and $a\leq 10^5$. For a given fourth power-free $Q\in\N$ one can define the set
$$
C_{Q}:=\{a\in\N:\;\cal{G}(a, Q)\subset \cal{V}_{f}\}
$$
and its counting function
$$
C_{Q}(x)=\#(C_{Q}\cap [1,x]).
$$
Because $(a, Q)$ is a solution of Problem \ref{mainprob} if and only if $(a, Q^{3})$ is a solution, we get that for any square-free $Q$ we have equality of sets $C_{Q}=C_{Q^{3}}$. Moreover, from the first part of Theorem \ref{bihomo} applied with $d=3$, we know that $C_{Q^2}$ is infinite for each $Q$.

For $Q\in\{1, \ldots, 14\}$ we have the following initial elements of the sets $C_{Q}$ (we omit $C_{8}$ due to the equality $C_{8}=C_{2}$):
\begin{equation*}
\begin{array}{|c|l|}
\hline
  Q & C_{Q} \\
\hline
  2 & 47, 69, 77, 79, 89, 94, 127, 138, 154, 155, 158, 171, 178, 188, 205, 219, 223, \ldots \\
  3 & 20, 31, 35, 37, 40, 47, 55, 60, 61, 73, 79, 92, 93, 95, 105, 111, 120, 127, 136, \ldots \\
  4 & 5, 14, 15, 20, 21, 31, 34, 37, 39, 46, 47, 49, 53, 55, 56, 60, 65, 66, 69, 73, \ldots\\
  5 & 8, 18, 19, 24, 29, 33, 40, 56, 79, 88, 90, 95, 98, 99, 104, 120, 126, 128, 129, \ldots\\
  6 & 15, 19, 33, 39, 40, 69, 73, 83, 85, 88, 90, 93, 95, 98, 104, 111, 114, 115, 129, \ldots\\
  7 & 51, 55, 67, 68, 73, 85, 89, 92, 113, 115, 120, 129, 136, 149, 155, 158, 179,  \ldots \\
  9 &14, 19, 20, 21, 24, 31, 34, 35, 37, 39, 40, 46, 47, 55, 60, 61, 65, 66, 67, 69,  \ldots \\
  10 &15, 24, 31, 34, 39, 65, 66, 69, 77, 89, 104, 111, 114, 129, 141, 143, 150, 156,\ldots \\
  11 &15, 18, 20, 24, 28, 31, 47, 53, 69, 79, 84, 95, 98, 104, 111, 113, 127, 133, 136,\ldots \\
  12 & 14, 34, 37, 46, 47, 61, 69, 85, 92, 94, 95, 111, 126, 143, 148, 154, 157, 158,\ldots\\
%  13 & 8, 18, 24, 40, 46, 51, 56, 63, 68, 88, 90, 93, 98, 104, 105, 120, 126, 127, 128,\ldots \\
%  14 & 21, 24, 29, 55, 65, 67, 79, 90, 92, 105, 109, 115, 120, 124, 127, 133, 141, 149,\ldots \\
  \hline
  \end{array}
\end{equation*}
\begin{center}
Table 2. Initial elements of the sets $C_{Q}$ for $Q=2,\ldots, 12, Q\neq 8$.
\end{center}

In the light of our computations one can ask the following

\begin{ques}
What is the order of magnitude of $C_{Q}(x)$?
\end{ques}

Our numerical computations suggest that, at least for $Q\in\{2, \ldots, 12\}$, the function $C_{Q}(x)/x$ has a limit (see picture below), say $L(Q)$. The conjectured estimations of $L(Q)$ can be easily read of from the values of $C_{Q}(10^5)$, which are as follows
\begin{equation*}
\begin{array}{|c|c|c|c|c|c|}
\hline
  Q           & 2       & 3      & 4       & 5       & 6     \\
\hline
 C_{Q}(10^5)  & 13623   & 17766  & 37872   & 17130   &  17456\\
\hline
\hline
  Q           & 7      & 9       & 10       & 11     & 12\\        \hline
 C_{Q}(10^5)  & 16120  & 38562   & 13458    & 17682  & 17292 \\
\hline
\end{array}
\end{equation*}
\begin{center}
Table 3. Values of $C_{Q}(10^{5})$ for $Q=2,\ldots, 12, Q\neq 8$.
\end{center}
In particular, $L(2)\approx 0.13623, L(3)\approx 0.17766, L(4)\approx 0.37872$ and so on.

Results of our computations strongly suggest the following:

\begin{conj}
Let $f(x,y)=(y^2-x^3)/x$.
\begin{enumerate}
\item For each $Q\in\N_{\geq 2}$ the set $C_{Q}$ is not empty.
\item For each $Q\in\N_{\geq 2}$ the set $C_{Q}$ is infinite.
\end{enumerate}
\end{conj}

\begin{rem}
{\rm From the first part of Theorem \ref{bihomo} applied for $d=3$ we know that $C_{Q^{2}}$ is infinite. The method of proof can be used to get a lower bound for $C_{Q^{2}}(x)$. Indeed, to make things explicit as much as possible, we take in the proof of Theorem \ref{bihomo}
$$
p_{0}=\frac{3}{4},\quad p_{1}=\frac{1}{4Q},\quad q_{0}=3(3Q^2-1)(3Q^2+1)u,\quad q_{1}=(1-3Q^2)(3Q^2+1)v.
$$
Then, the corresponding value of $a$ is given by
$$
a=a(u,v)=768Q^3 (3 Q^2-1)^2 (3 Q^2+1)^2 (3 Q u^2-v^2)(3 Q^3 v^2-u^2).
$$
Now, if we fix $u\in\N$, then for each integer
$$
v\in \left(\frac{1}{Q\sqrt{3Q}}u,\sqrt{3Q}u\right)
$$
we have $a(u, v)>0$. It is possible that for some $u, v\in\N_{+}$ the value $a(u,v)$ is a fourth power. However, the number of such possibilities is finite because the (projective) curve $z^4=a(u,v)$ is of genus 3, and has only finitely many primitive solutions. As a consequence we get that
$$
C_{Q^{2}}(x)>c(Q^2)x^{1/4}
$$
for some computable constant $c(Q^2)$. However, this bound is far weaker than our expectations.
}
\end{rem}
\begin{figure}[h]\label{picture} %  figure placement: here, top, bottom, or page
       \centering
         \includegraphics[width=4.5in]{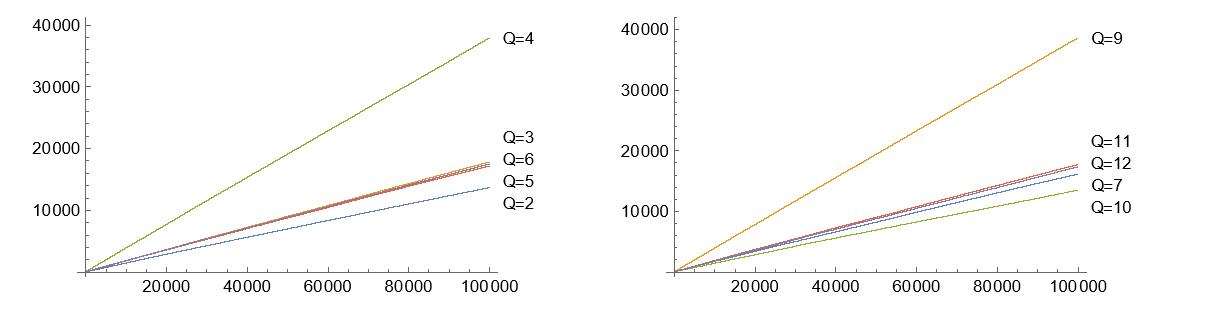}
        \caption{Plots of the functions $C_{Q}(x)$ for $Q=2, \ldots 6$ (left) and $Q=7,\ldots, 12, Q\neq 8$, (right) and $x\leq 10^5$.}
       \label{fig:pic1}
    \end{figure}

\begin{figure}[h]\label{picture} %  figure placement: here, top, bottom, or page
       \centering
         \includegraphics[width=4.5in]{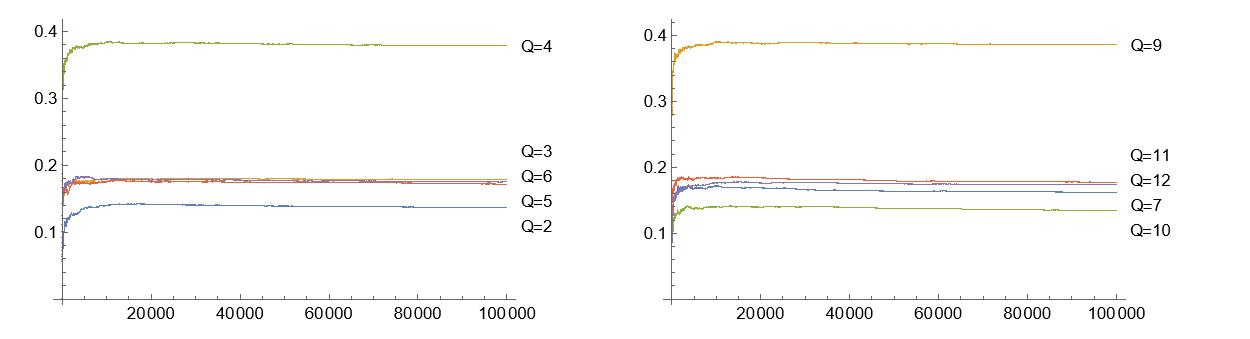}
        \caption{Plots of the functions $C_{Q}(x)/x$ for $Q=2, \ldots 6$ (left) and $Q=7,\ldots, 12, Q\neq 8$, (right) and $x\leq 10^5$.}
       \label{fig:pic2}
    \end{figure}

It is also quite interesting to note that
\begin{align*}
\bigcap_{Q=2}^{12}C_{Q}\cap [1,10^5]=\{&32504, 34023, 36333, 40399, 41080, 41415, 55260, \\
                                        &55965, 73242, 86415, 99342\},
\end{align*}
 i.e., for each $Q\in\{2,\ldots, 12\}, i\in\N$ and each $a$ from the above set, the elliptic curve $E_{i}(a, Q):\;y^2=x^3+aQ^{i}x$ has positive rank. An additional check confirms that
\begin{align*}
\bigcap_{Q=2}^{13}C_{Q}\cap [1,10^5]&=\{32504, 34023, 36333, 40399, 41080, 55260, 73242, 99342\},\\
\bigcap_{Q=2}^{14}C_{Q}\cap [1,10^5]&=\{40399, 73242, 99342\}.
\end{align*}
However, the set $\bigcap_{Q=2}^{15}C_{Q}\cap [1,10^5]$ is empty. Our computations suggest the following:

\begin{ques}
Let $N\in\N_{+}$ be given. Is the set $\bigcap_{Q=2}^{N}C_{Q}$ non-empty?
\end{ques}

We expect that the answer on the question above is YES. On the other hand, we have the following:

\begin{prop}
We have $\bigcap_{Q=2}^{\infty}C_{Q}=\emptyset$.
\end{prop}
\begin{proof}
Let us suppose that the statement is not true and take $a\in\bigcap_{Q=2}^{\infty}C_{Q}$. Without loss of generality we can assume that $a$ is fourth power-free. Write $a=q_{1}q_{2}^2q_{3}^{3}$, where $q_{1}, q_{2}, q_{3}$ are square-free pairwise co-prime integers. Let us take $Q=q_{1}^{3}q_{2}^2q_{3}$ and observe that $a\not\in C_{Q}$. Indeed, we note that the curve $E_{1}(a, Q):\;y^2=x^3+aQx$ is isomorphic over $\Q$ with the curve $E:\;y^2=x^3+x$. However, the rank of $E(\Q)$ is zero and the same is true for $E_{1}(a, Q)(\Q)$ - a contradiction. 
\end{proof}

Next, we note that for given $Q_{1}, Q_{2}\in\{2,\ldots 11\}, Q_{1}\neq Q_{2}$, our computations suggest that for sufficiently large $x$, the difference $C_{Q_{1}}(x)-C_{Q_{2}}(x)$ has a constant sign, and in fact there is  $C_{Q_{1}}(x)-C_{Q_{2}}(x) \rightarrow \pm\infty$ with $x$ going to infinity.

In the light of our computations one can formulate the following general problem.

\begin{prob}
Let $\cal{Q}=\{Q\in\N_{+}:\; Q\;\mbox{is fourth power free}\}$.
\begin{enumerate}
\item
Characterize pairs $Q_{1}, Q_{2}\in\cal{Q}, Q_{1}>Q_{2}$ of positive integers such that the sign of the function $C_{Q_{1}}(x)-C_{Q_{2}}(x)$ is eventually constant.
\item Does there exist $Q_{1}, Q_{2}\in\cal{Q}$ such that $C_{Q_{1}}(x)-C_{Q_{2}}(x)$ changes the sign infinitely often?
\end{enumerate}
\end{prob}

\bigskip

In the light of Theorem \ref{fixeda}, for given $a\in\cal{R}$ one can define the quantity
$$
m(a)=\op{min}\{Q\in\N_{\geq 2}:\;(a, Q)\;\mbox{is a proper solution of Problem \ref{mainprob}}\}.
$$
Although it is not clear that the function is well defined, results of our computations suggest that this is the case. However, as we observed, the corresponding value of $Q$ computed from rational point on the elliptic curve $\cal{E}_{a}(p, v)$ constructed in Theorem \ref{fixeda} is much bigger than it is expected. More precisely, from our computations it follows that for each positive $a\in\cal{R}\cap [1, 10^4]$, the value of $m(a)$ exists. The values of $m(a)$ for $a\in \cal{R}\cap [3, 60]$ are given in Table 4 below.

\begin{equation*}
\begin{array}{|c|ccccccccccccccc|}
\hline
a   & 3 & 5 & 8 & 9 & 13 & 14 & 15 & 18 & 19 & 20 & 21 & 24 & 28 & 29 & 31 \\
 \hline
m(a) & 17 & 17 & 5 & 17 & 15 & 12 & 6 & 5 & 5 & 3 & 14 & 5 & 11 & 5 & 3 \\
 \hline
 \hline
a &      33 & 34 & 35 & 37 & 39 & 40 & 46 & 47 & 48 & 49 & 51 & 53 & 55 & 56 & 60 \\
\hline
m(a) &    5 & 10 & 3 & 3 & 6 & 3 & 12 & 2 & 17 & 17 & 7 & 11 & 3 & 5 & 3 \\
\hline
\end{array}
\end{equation*}
\begin{center}
Table 4. The values of $m(a)$ for $a\in\cal{R}\cap [3, 60]$.
\end{center}

In fact, we have
$$
\{m(a):\; a\in\cal{R}\cap [1, 10^4]\}=\{2, 3, 5, 6, 7, 10, 11, 12, 13, 14, 15, 17, 21, 22, 23, 26\}.
$$
Let us write $\cal{R}=\{a_{1}, a_{2}, \ldots\}$, where $a_{i}<a_{i+1}$ for $i\in\N_{+}$. On Figure 4 below, we present the plot of the function $m(a_{n})$ for $n\in\{1,\ldots, 6005\}$, i.e., $a_{n}\in\cal{R}\cap [1, 10^4]$.
\begin{figure}[h]\label{picture} %  figure placement: here, top, bottom, or page
       \centering
         \includegraphics[width=5in]{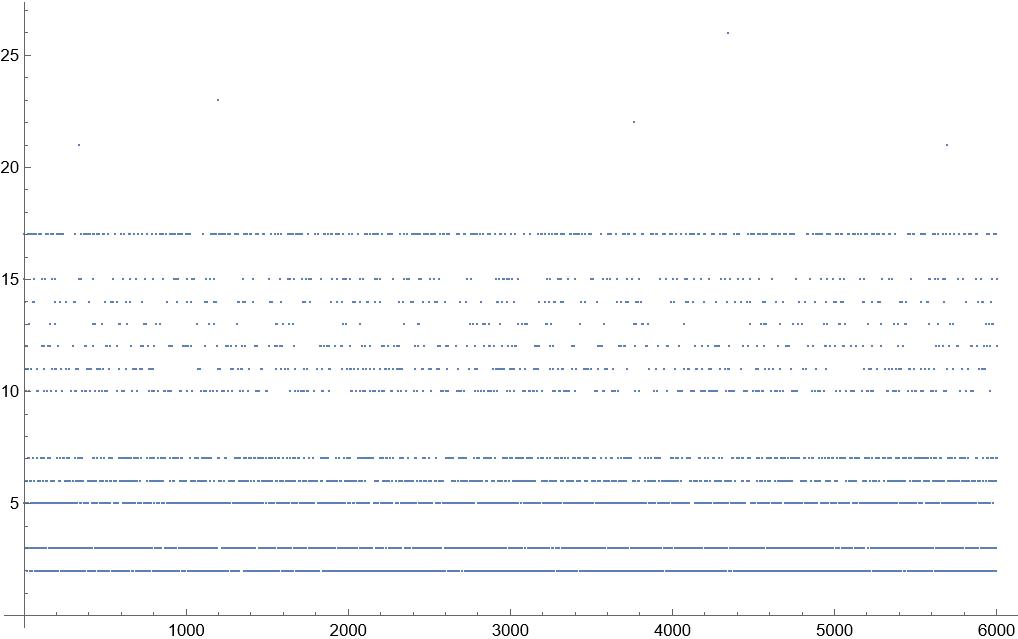}
        \caption{Plot of the function $f(n)=m(r_{n})$, where $n\leq 6005$.}
       \label{fig:disc1}
    \end{figure}

In the table below, for each element $Q$ of the above set we present the smallest value of $a\in\cal{R}$ such that $m(a)=Q$.

\begin{equation*}
\begin{array}{|c|cccccccc|}
\hline
 m(a) &   2 & 3  & 5 & 6 & 7  & 10 & 11 & 12   \\
\hline
  a   &  47 & 20 & 8 & 15& 51 & 34 & 28 & 14 \\
\hline\hline
 m(a)  & 13  & 14  & 15 & 17& 21 & 22 & 23 & 26 \\
\hline
   a   & 63  &21   & 13 & 3 & 594& 6295 & 2028 & 7255  \\
\hline
\end{array}
\end{equation*}
\begin{center}
Table 4. The smallest value of $a$ such that $m(a)=Q$, where $Q\in \{m(a):\;  a\in\cal{R}\cap [1, 10^4]\}$.
\end{center}

We formulate the following:

\begin{conj}\label{mdefined}
The function $m:\;\cal{R}\rightarrow \{Q\in\N_{+}:\;Q\;\mbox{is fourth power-free}\}$ is well defined.
\end{conj}

It is clear that Conjecture \ref{fixedacurve} implies Conjecture \ref{mdefined}. However, it is not clear whether there is an implication in the other side.

\begin{conj}\label{mlimitpoints}
The function $m:\;\cal{R}\rightarrow \{Q\in\N_{+}:\;Q\;\mbox{is fourth power-free}\}$ is onto and each element in its image is a limit point. In particular
\begin{equation*}
\liminf_{a\in\cal{R}}m(a)=2,\quad \limsup_{a\in\cal{R}}m(a)=+\infty.
\end{equation*}
\end{conj}

One can also ask about average behaviour of the function $m(a_{n})$. On Figure 5 we plot the function
$$
\cal{A}(x)=\frac{\sum_{i\leq x}m(a_{i})}{x}
$$
for $x\leq 6005$. Based on our computations we dare to formulate the following:

\begin{conj}
If Conjecture \ref{mdefined} is true, then the function $\cal{A}$ is bounded.
\end{conj}

\begin{figure}[h]\label{picture} %  figure placement: here, top, bottom, or page
       \centering
         \includegraphics[width=5in]{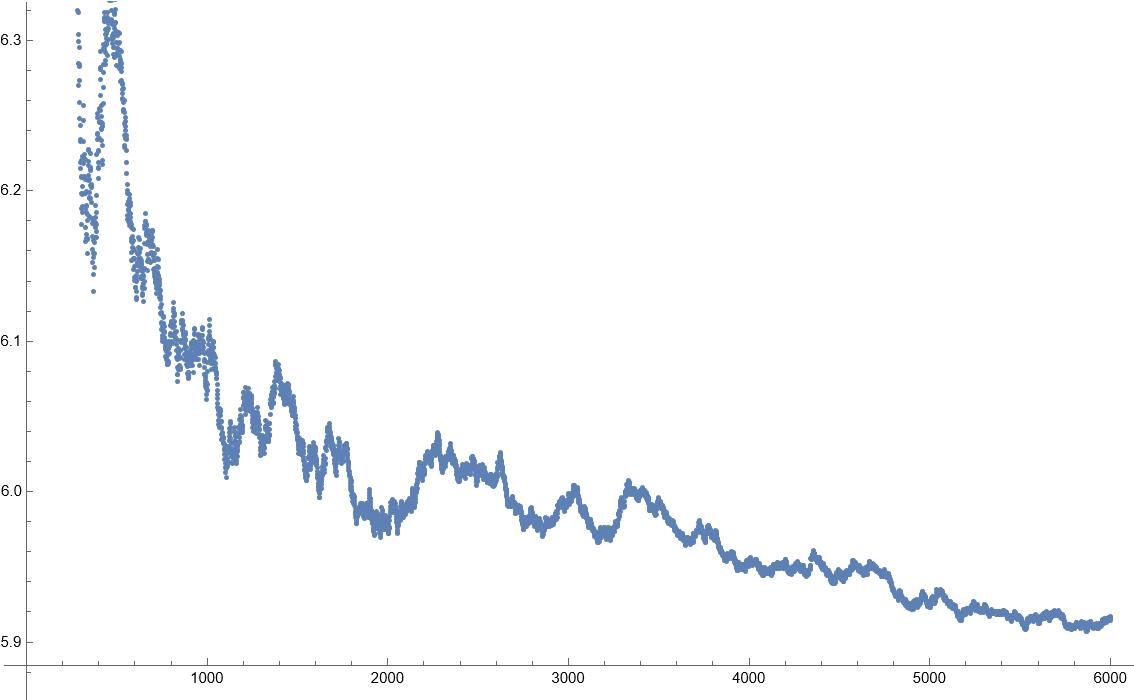}
        \caption{Plot of the function $\cal{A}(x)$, where $x\leq 6005$.}
       \label{fig:disc1}
    \end{figure}

\bigskip

In the light of our findings concerning the questions around Problem \ref{mainprob} for $f(x, y)=(y^2-x^3)/x$, one can ask and investigated similar questions in the case of the polynomial $f(x,y)=y^2-x^3$, both form computational and theoretical point of view. To be more precise, we are interested in the existence of proper solutions of Problem \ref{mainprob}, i.e., $Q$ is neither as quare nor a cube and for each $i\in\N$ the rank of elliptic curve
$$
E'_{i}(a, Q):\;y^2=x^3+aQ^{i}
$$
is positive. We show that $(3, 2)$ is a proper solution of Problem \ref{mainprob}. To see this, one can compute the following data:
\begin{equation*}
\begin{array}{|c|c|l|}
\hline
  i & r(E'_{i}(3, 2)) & \mbox{generators of}\;E'_{i}(3, 2)(\Q) \\
\hline
  0 & 1 & (-2, 5) \\
  1 & 1 & (1/4, 65/8) \\
  2 & 1 & (4, 14)  \\
  3 & 1 & (-2, -16) \\
  4 & 2 & (16, 68), (-8, 4) \\
  5 & 1 & (48217/5041, -15728083/357911) \\
\hline
\end{array}
\end{equation*}
\begin{center}
Table 5. Ranks and the generators for the elliptic curves $E_{i}'(3, 2):\;y^2=x^3+3\cdot2^{i}, i=0, \ldots, 5$.
\end{center}

Consequently, the pair $(32, 2)$ is a proper solution of Problem \ref{mainprob} for $f(x, y)=y^2-x^3$.

We formulate the following:

\begin{prob}
Investigate properties of the set of proper solutions of Problem \ref{mainprob} for $f(x, y)=y^2-x^3$. In particular prove the following:
\begin{enumerate}
\item For each $a\in\Z$ such that the elliptic curve $y^2=x^3+a$ has positive rank, there are infinitely many values of $Q\in\Z$ such that $(a, Q)$ is a proper solution of Problem \ref{mainprob}.
\item For each $Q\in\Z\setminus\{-1,0,1\}$, there are infinitely many values of sixth power-free integers $a$ such that $(a, Q)$ is a proper solution of Problem \ref{mainprob}.
\end{enumerate}
\end{prob}

%{\bf TO DO}

%Numerical computations suggest that for many values of $t$ we have $\cal{G}(8(t^2-t+1),5)\subset \cal{V}_{f}$, where $f(x,y)=(y^2-x^3)/x$.
%\begin{ques}
%Is it true, that there exist quadratic base change $t\mapsto\phi(t)$ such that  $\cal{G}(8(\phi(t)^2-\phi(t)+1),5)\subset \cal{V}_{f}$?
%\end{ques}
%Of course, one can formulate similar question with whether for a given $a$ there is $Q$ such that the rank of $E_{i}:\;y^2=x^3+aQ^{i}x$ is $>0$ for $i=0, 1, 2, 3$. However, it %is clear that there are infinitely many values of $a$ such that $E_{0}$ has rank 0 and thus $\cal{V}_{f}$ is finite. However, we still can formulate the following
%\begin{ques}
%Let $f(x,y)=(y^2-x^3)/x$. Let $a\in\N_{+}$ be given and suppose that the elliptic curve $E_{0}:\;y^2=x^3+ax$ has positive rank. Does there exist an $Q\in\N_{\geq 2}$ such that %$\cal{G}(a,Q)\subset\cal{V}_{f}$?
%\end{ques}

\end{document}